\documentclass[final,reqno]{siamltex}
\usepackage{latexsym,amsmath,amssymb,amsfonts,mathrsfs}
\usepackage{epsf,graphicx,epsfig,epstopdf,color,cite,cases}
\usepackage{subfigure,graphics,multirow,marginnote,graphicx}
\newcommand{\R}{{\mathbb R}}

\newcommand{\Sp}{{\mathbb S}}
\newcommand{\ds}{\displaystyle}
\newcommand{\no}{\nonumber}
\newcommand{\be}{\begin{eqnarray}}
\newcommand{\ben}{\begin{eqnarray*}}
\newcommand{\en}{\end{eqnarray}}
\newcommand{\enn}{\end{eqnarray*}}
\newcommand{\ba}{\backslash}
\newcommand{\pa}{\partial}

\newcommand{\ov}{\overline}

\newcommand{\I}{{\rm Im}}
\newcommand{\Rt}{{\rm Re}}

\newcommand{\G}{\Gamma}

\newcommand{\wid}{\widetilde}

\newcommand{\ra}{\rightarrow}
\newcommand{\se}{\setminus}

\newtheorem{remark}[theorem]{Remark}

\begin{document}
\renewcommand{\theequation}{\arabic{section}.\arabic{equation}}
%
\title{\bf Uniqueness in inverse scattering problems with phaseless far-field data at a fixed frequency. II
}
\author{Xiaoxu Xu\thanks{Academy of Mathematics and Systems Science, Chinese Academy of Sciences,
Beijing 100190, China and School of Mathematical Sciences, University of Chinese
Academy of Sciences, Beijing 100049, China ({\tt xuxiaoxu14@mails.ucas.ac.cn})}
\and Bo Zhang\thanks{LSEC and Academy of Mathematics and Systems Science, Chinese Academy of
Sciences, Beijing, 100190, China and School of Mathematical Sciences, University of Chinese
Academy of Sciences, Beijing 100049, China ({\tt b.zhang@amt.ac.cn})}
\and Haiwen Zhang\thanks{Academy of Mathematics and Systems Science, Chinese Academy of Sciences,
Beijing 100190, China ({\tt zhanghaiwen@amss.ac.cn})}}
\date{}

\maketitle

\begin{abstract}
This paper is concerned with uniqueness in inverse acoustic scattering with phaseless far-field data
at a fixed frequency. In our previous work ({\em SIAM J. Appl. Math. \bf78} (2018), 1737-1753), 
by utilizing spectral properties of the far-field operator we proved for the first time that the obstacle 
and the index of refraction of an inhomogeneous
medium can be uniquely determined by the phaseless far-field patterns generated by infinitely many sets
of superpositions of two plane waves with different directions at a fixed frequency under the a priori
assumption that the obstacle is known to be a sound-soft or non-absorbing impedance obstacle and the
index of refraction $n$ of the inhomogeneous medium is real-valued and satisfies that either $n-1\ge c_1$
or $n-1\le-c_1$ in the support of $n-1$ for some positive constant $c_1$.
In this paper, we remove the a priori assumption on the obstacle and the index of refraction of the
inhomogeneous medium by adding a reference ball to the scattering system together with a simpler method
of using Rellich's lemma and Green's representation formula for the scattering solutions.
Further, our new method is also used to prove uniqueness in determining a locally
rough surface from the phaseless far-field patterns corresponding to infinitely many sets of superpositions
of two plane waves with different directions as the incident fields at a fixed frequency.
\end{abstract}



\pagestyle{myheadings}
\thispagestyle{plain}
\markboth{X. Xu, B. Zhang, and H. Zhang}{Uniqueness in inverse scattering with phaseless far-field data}

\section{Introduction}\label{sec1}

Inverse scattering theory has wide applications in such fields as radar, sonar, geophysics, medical imaging,
and nondestructive testing (see, e.g., \cite{CK,CK11,KG}).
This paper is concerned with inverse scattering with phaseless far-field data generated by incident plane waves.
We will consider three different scattering models: the scattering problems by a bounded obstacle,
an inhomogeneous medium, and a locally perturbed infinite plane surface
(called a locally rough surface in this paper).

Inverse scattering problems with full data have been widely studied both mathematically
and numerically over the past decades (see, e.g., \cite{CK,CK11,KG} and the references quoted there).
However, in many applications, the phase of the near-field or far-field pattern can not be measured
accurately compared with its modulus or intensity. Therefore, it is often desirable to recover the obstacle
and the medium from the modulus or intensity of the near-field or the far-field pattern
(or the phaseless near-field data or the phaseless far-field data).
Inverse scattering problems with phaseless near-field data, also called the (near-field)
{\em phase retrieval problems} in optics and other physical and engineering sciences, have been widely
studied numerically for the past decades (see, e.g. \cite{BaoLiLv2012,BaoZhang16,CH16,CH17,CFH17,K14,KR16,KR17}).
Recently, several results are available for the mathematical study (such as uniqueness and stability) of
the near-field phase retrieval problem (see, e.g., \cite{K14,K17,KR17,MH17,N15,N16}).

In contrast to the case with phaseless near-field data, inverse scattering with
phaseless far-field data is more challenging due to the {\em translation invariance property}
of the phaseless far-field pattern (see \cite{KR,LS04,XZZ,ZZ01}).
Thus, it is impossible to reconstruct the location of the scatterers from the phaseless far-field pattern
with one plane wave as the incident field. Nevertheless, several reconstruction algorithms have been
developed to reconstruct the shape of the scatterer from the phaseless far-field data with one plane wave
as the incident field (see \cite{ACZ16,I2007,IK2010,IK2011,KR,L16,LiLiu15,LLW17,S16}).
Further, by assuming a priori the scatterer to be a sound-soft ball centered at the origin, uniqueness
was established in determining the radius of the ball from a single phaseless far-field datum in \cite{LZ10}.
By studying the high frequency asymptotics of the far-field pattern, it was proved in \cite{majda76}
that the shape of a general smooth convex sound-soft obstacle can be recovered from
the modulus of the far-field pattern associated with one plane wave as the incident field.
However, for plane wave incidence no uniqueness results are available for the inverse problem of
recovering the scatterers from phaseless far-field data generated by one incident plane wave.

Recently, it was suggested in \cite{ZZ01} to use superpositions of two plane waves rather than one plane wave
as the incident fields with an interval of frequencies to break the translation invariance property of the
phaseless far-field pattern. A recursive Newton iteration algorithm in frequencies was further developed
in \cite{ZZ01} to recover both the location and the shape of the obstacle simultaneously
from multi-frequency phaseless far-field data. The idea was also extended to inverse scattering by
locally rough surfaces with phaseless far-field data in \cite{ZZ02}.
On the other hand, a fast imaging algorithm was developed in \cite{ZZ03} to numerically recover the
scattering obstacles from phaseless far-field data at a fixed frequency.
Furthermore, it was rigorously proved in \cite{XZZ} that the obstacle and the index of refraction of
an inhomogeneous medium can be uniquely determined by the phaseless far-field patterns generated by
infinitely many sets of superpositions of two plane waves at a fixed frequency
under {\em the a priori assumption} that the obstacle is a sound-soft obstacle or an impedance obstacle with
a real-valued impedance function and the refractive index $n$ of the inhomogeneous medium is real-valued
and satisfies the condition that either $n-1\ge c_1$ or $n-1\le-c_1$ in the support
of $n-1$ for some positive constant $c_1$.
The purpose of the present paper is to remove the a priori assumption on the obstacle and the refractive
index of the inhomogeneous medium by adding a known reference ball to the scattering system
in conjunction with a simple technique based on Rellich's lemma and Green's representation formula for
the scattering solutions.

It should be remarked that the idea of adding a reference ball to the scattering system was first used
in \cite{LLZ09} to improve the linear sampling method numerically. It was recently used in \cite{ZG18}
to prove uniqueness results for inverse scattering with phaseless far-field data.
Precisely, it was proved in \cite{ZG18} that an obstacle $D$ and its boundary condition
can be uniquely determined by the phaseless far-field data
\ben
&&\{|u^\infty_{D\cup B}(\hat{x},d_0)|:\hat{x}\in\Sp^2\},\;\;
\{|v^\infty_{D\cup B}(\hat{x},z)|:\hat{x}\in\Sp^2,z\in\Pi\},\\
&&\{|u^\infty_{D\cup B}(\hat{x},d_0)+v^\infty_{D\cup B}(\hat{x},z)|:\hat{x}\in\Sp^2,z\in\Pi\}
\enn
for a fixed wave number $k>0$ and a fixed $d_0\in\Sp^2$, where $\Sp^2:=\{d\in\R^3\,:\,|d|=1\}$
denotes the unit sphere in $\R^3$, $B$ is a reference sound-soft ball,
$u^\infty_{D\cup B}(\hat{x},d_0)$ and $u^\infty_{D\cup B}(\hat{x},z)$
denote the far-field pattern of the scattering solution of the Helmholtz equation $\Delta u+k^2u=0$
associated with the obstacle $D\cup B$ corresponding to the incident plane wave $e^{ikx\cdot d_0}$
and the incident point source ${e^{ik|x-z|}}/(4\pi|x-z|)$ ($x\not=z$), respectively,
and $\Pi$ is the boundary of a simply-connected convex polyhedron $P$ such that
$P\subset\R^3\se\ov{D\cup B}$ and $k^2$ is not a Dirichlet eigenvalue of $-\Delta$ in $P$.
A similar uniqueness result was also proved in \cite{ZG18} for the inhomogeneous medium case
by adding a reference homogeneous medium ball.

The remaining part of the paper is organized as follows. In section \ref{sec2}, we introduce
the direct and inverse scattering problems considered in this paper.
Sections \ref{o} and \ref{m} are devoted to the inverse obstacle and medium scattering problems,
respectively, while Section \ref{lrs} is devoted the inverse rough surface scattering problem.
Conclusions are given in Section \ref{con}.

\section{The direct scattering problems}\label{sec2}

To give a precise description of the scattering problems, we assume that $D$ is an open and bounded domain
in $\R^3$ with $C^2$-boundary $\G:=\pa D$ satisfying that the exterior $\R^3\se\ov{D}$ of $\ov{D}$ is connected.
Note that $D$ may not be connected and thus may consist of several (finitely many) connected components.
Suppose the time-harmonic ($e^{-i\omega t}$ time dependence) plane wave
\ben
u^i=u^i(x,d):=e^{ikx\cdot d}
\enn
is incident on the obstacle $D$ from the unbounded part $\R^3\se\ov{D}$.
Here, $d\in\Sp^2$ is the incident direction,
$k=\omega/c>0$ is the wave number, $\omega$ and $c$ are the wave frequency and speed in the homogeneous
medium in $\R^3\ba\ov{D}$. Let $u^s$ be the scattered field. Then the total field $u:=u^i+u^s$ outside
an impenetrable obstacle $D$ satisfies the exterior boundary value problem:
\begin{subequations}
\be\label{he}
\Delta u+k^2u &=& 0\quad{\rm in}\;\;\R^3\se\ov{D},\\ \label{bc}
{\mathscr B}u &=& 0\quad{\rm on}\;\;\pa D,\\ \label{rc}
\lim\limits_{r\rightarrow\infty}r\left(\frac{\pa u^s}{\pa r}-iku^s\right)&=&0,\quad r=|x|,
\en
\end{subequations}
where (\ref{he}) is the Helmholtz equation and (\ref{rc}) is the Sommerfeld radiation condition.
The boundary condition ${\mathscr B}$ in (\ref{bc}) depends on the physical property of the obstacle $D$:
\ben
\left\{\begin{array}{ll}
{\mathscr B}u=u  & \text{if $D$ is a sound-soft obstacle},\\
\ds {\mathscr B}u=\frac{\pa u}{\pa\nu}+\eta u  & \text{if $D$ is an impedance obstacle},\\
\ds {\mathscr B}u=u\;\;\text{on}\;\;\G_D,\;\;{\mathscr B}u=\frac{\pa u}{\pa\nu}+\eta u\;\;\text{on}\;\;\G_I &
\text{if $D$ is a partially coated obstacle},
\end{array}\right.
\enn
where $\nu$ is the unit outward normal to the boundary $\pa D$ or $\G_I$ and $\eta$ is
the impedance function satisfying that $\I[\eta(x)]\geq0$ for all $x\in\pa D$ or $\G_I$.
In this paper, we assume that $\eta\in C(\pa D)$ or $C(\G_I)$, that is, $\eta$ is continuous on $\pa D$ or $\G_I$.
When $\eta=0$, the impedance boundary condition becomes the Neumann boundary condition (a sound-hard obstacle).
For a partially coated obstacle $D$, we assume that the boundary $\G:=\pa D$ has a Lipschitz dissection
$\G=\G_D\cup\Pi\cup\G_I$, where $\G_D$ and $\G_I$ are disjoint, relatively open subsets of $\G$,
having $\Pi$ as their common boundary in $\G$ (see, e.g., \cite{M}). Furthermore, boundary conditions of
Dirichlet and impedance type are specified on $\G_D$ and $\G_I$, respectively.

The problem of scattering of a plane wave by an inhomogeneous medium is modelled
by the medium scattering problem:
\begin{subequations}
\be\label{he-n}
\Delta u+k^2n(x)u&=&0\quad\mbox{in}\;\;\R^3,\\ \label{rc-n}
\lim\limits_{r\rightarrow\infty}r\left(\frac{\pa u^s}{\pa r}-iku^s\right)&=&0,\quad r=|x|,
\en
\end{subequations}
where $u=u^i+u^s$ is the total field, $u^s$ is the scattered field, and $n$ in the reduced wave
equation (\ref{he-n}) is the index of refraction characterizing the inhomogeneous medium.
In this paper, we assume that $m:=n-1$ has the compact support $\ov{D}$ and $n\in L^\infty(D)$ with
$\Rt[n(x)]>0$ and $\I[n(x)]\geq0$ for almost all $x\in D$.

Well-posedness of the problems (\ref{he})-(\ref{rc}) and (\ref{he-n})-(\ref{rc-n}) has been established
in \cite{CK,KG,Kirsch} by using either the integral equation method or the variational method
(see Theorem 3.11 in \cite{CK} and Theorem 1.1 in \cite{KG} for the exterior Dirichlet problem,
Theorem 2.2 in \cite{KG} for the exterior impedance problem, Theorem 8.5 in \cite{CC14} 
for the partially coated problem, and Theorem 6.9 in \cite{Kirsch} for the medium scattering 
problem (\ref{he-n})-(\ref{rc-n})).
In particular, it is well-known that the scattered field $u^s$ has the asymptotic behavior:
\ben
u^s(x,d)=\frac{e^{ik|x|}}{|x|}\left\{u^\infty(\hat{x},d)+\left(\frac{1}{|x|}\right)\right\},
\quad|x|\rightarrow\infty
\enn
uniformly for all observation directions $\hat{x}=x/|x|\in\Sp^2$, where $u^\infty(\hat{x},d)$ is the
far field pattern of $u^s$ which is an analytic function of $\hat{x}\in\Sp^2$ for each $d\in\Sp^2$
and of $d\in\Sp^2$ for each $\hat{x}\in\Sp^2$ (see, e.g., \cite{CK}).

The last model we consider is the scattering of acoustic plane waves by a locally perturbed plane surface.
Precisely, let $h\in C^2(\R^2)$ with a compact support in $\R^2$.
Denote by $\G:=\{x=(x_1,x_2,x_3):x_3=h(x_1,x_2),(x_1,x_2)\in\R^2\}$
the locally rough surface and by $D_+:=\{x=(x_1,x_2,x_3):x_3>h(x_1,x_2),(x_1,x_2)\in\R^2\}$
the half-plane above the locally rough surface $\G$ which is filled with a homogeneous medium.
Let $k=\omega/c>0$ be the wave number in $D_+$ with $\omega$ and $c$ being the wave frequency and speed,
respectively. Assume that the time-harmonic ($e^{-i\omega t}$ time dependence) plane wave
\[u^i=u^i(x,d)=e^{ikx\cdot d}\]
is incident on the locally rough surface $\G$ from $D_+$, where $d\in\Sp^2_-$ is the incident direction
and $\Sp^2_-:=\{x=(x_1,x_2,x_3)\in\R^3:|x|=1,x_3<0\}$ is the lower part of the unit sphere $\Sp^2$. 
Then the total field $u$ satisfies the Helmholtz equation
\be\label{he-l}
\Delta u+k^2u=0\;\;\text{in}\;\;D_+.
\en
Here, the total field $u=u^i+u^r+u^s$ satisfies the Dirichlet boundary condition
\be\label{Dbcl}
u=u^i+u^r+u^s=0\;\;\text{on}\;\;\G
\en
for a sound-soft surface or the Neumann boundary condition
\be\label{Nbcl}
\frac{\pa u}{\pa\nu}=\frac{\pa u^i}{\pa\nu}+\frac{\pa u^r}{\pa\nu}+\frac{\pa u^s}{\pa\nu}=0\;\;\text{on}\;\;\G
\en
for a sound-hard surface, where $\nu$ is the unit outward normal vector on $\G$ directed into $D_+$,
$u^r$ is the reflected wave by the infinite plane $x_3=0$:
\[u^r(x,d):=\begin{cases}
-e^{ikx\cdot d'} & \text{for the Dirichlet boundary condition on}\;\;\G,\\
e^{ikx\cdot d'} & \text{for the Neumann boundary condition on}\;\;\G
\end{cases}\]
with $d'=(d_1,d_2,-d_3)\in\Sp^2$ and $u^s$ is the unknown scattered field to be determined.
The scattered field $u^s$ is required to satisfy the Sommerfeld radiation condition
\be\label{rcl}
\lim_{r\ra\infty}r\left(\frac{\pa u^s}{\pa\nu}-iku^s\right)=0,\;\;r=|x|,\;\;x\in D_+.
\en

The problem (\ref{he-l})-(\ref{Dbcl}) and (\ref{rcl}) and the problem (\ref{he-l}) and (\ref{Nbcl})-(\ref{rcl})
model the problem of scattering of acoustic plane waves by a two-dimensional sound-soft surface and
a two-dimensional sound-hard surface, respectively.
In the case when the local perturbation is invariant in the $x_2$-direction,
these two problems can be reduced to the two-dimensional case, which also model the problem of
electromagnetic plane waves by a locally perturbed, perfectly reflecting, infinite plane in
the transverse-electric and transverse-magnetic polarization cases, respectively.

The existence of a unique solution to the problem (\ref{he-l})-(\ref{Dbcl}) and (\ref{rcl})
has been proved in \cite{W} by the integral equation method,
while the existence of a unique solution to the problem (\ref{he-l})-(\ref{Dbcl}) and (\ref{rcl})
can also be established similarly by using the integral equation method (see \cite[Remark 4.7 (i)]{W}).
We remark that, in the two-dimensional case, the well-posedness of the problem (\ref{he-l})-(\ref{Dbcl})
and (\ref{rcl}) was proved in \cite{ZZ13} by a new integral equation formulation and in \cite{BaoLin2011}
by the variational method, while the existence and uniqueness of solutions to the problem (\ref{he-l}) and (\ref{Nbcl})-(\ref{rcl}) was established in \cite{QZZ} by the integral equation method.
From the results in \cite{W} it is known that $u^s$ has the following asymptotic behavior at infinity
\ben
u^s(x,d)=\frac{e^{ik|x|}}{|x|}\left[u^\infty(\hat x,d)+O\left(\frac1{|x|}\right)\right],
\;\;|x|\ra\infty
\enn
uniformly for all observation directions $\hat{x}\in\Sp^2_+$ with $\Sp^2_+:=\{x=(x_1,x_2,x_3):|x|=1,x_3>0\}$
the upper part of the unit sphere $\Sp^2$,
where $u^\infty(\hat{x},d)$ is the far-field pattern of the scattered field $u^s$ which is an analytic
function of $\hat{x}\in\Sp^2_+$ for each $d\in\Sp^2_-$ and of $d\in\Sp^2_-$ for each $\hat{x}\in\Sp^2_+$.

Throughout this paper, we assume that the wave number $k$ is arbitrarily fixed. Following \cite{XZZ,ZZ01,ZZ03},
we make use of the following superposition of two plane waves as the incident field:
\ben
u^i=u^i(x;d_1,d_2)=u^i(x,d_1)+u^i(x,d_2)=e^{ikx\cdot d_1}+e^{ikx\cdot d_2},
\enn
where the incident directions $d_1,d_2\in\Sp^2$ for the obstacle and medium
scattering problems and $d_1,d_2\in\Sp^2_-$ for the locally rough surface scattering problems. 
Then the scattered field $u^s$ has the asymptotic behavior
\ben
u^s(x;d_1,d_2)=\frac{e^{ik|x|}}{|x|}\left\{u^\infty(\hat{x};d_1,d_2)
+O\left(\frac{1}{|x|}\right)\right\},\;\;|x|\to\infty
\enn
uniformly for all observation directions $\hat{x}\in\Sp^2$ for the obstacle and medium
scattering problems or the asymptotic behavior
\ben
u^s(x;d_1,d_2)=\frac{e^{ik|x|}}{|x|}\left\{u^\infty(\hat{x};d_1,d_2)
+O\left(\frac{1}{|x|}\right)\right\},\;\;|x|\to\infty
\enn
uniformly for all observation directions $\hat{x}\in\Sp^2_+$ for the locally rough surface scattering problems.
From the linear superposition principle it follows that
\[u^s(x;d_1,d_2)=u^s(x,d_1)+u^s(x,d_2)\] and
\ben
u^\infty(\hat{x};d_1,d_2)=u^\infty(\hat{x},d_1)+u^\infty(\hat{x},d_2),
\enn
where $u^s(x,d_j)$ and $u^\infty(\hat{x},d_j)$ are the scattered field and its far-field pattern
corresponding to the incident plane wave $u^i(x,d_j)$, respectively, $j=1,2$.

The {\em inverse obstacle or medium scattering problem} we consider in this paper is to reconstruct
the obstacle $D$ and its physical property or the index of refraction $n$ of the inhomogeneous medium
from the phaseless far-field pattern $|u^\infty(\hat{x};d_1,d_2)|$ for $\hat{x},d_j\in\Sp^2$, $j=1,2$,
while the {\em inverse rough surface scattering problem} considered is to recover the locally rough surface
$\G$ from the phaseless far-field pattern $|u^\infty(\hat{x};d_1,d_2)|$ for $\hat{x}\in\Sp_+^2$,
$d_j\in\Sp^2_-$, $j=1,2$. The purpose of this paper is to prove uniqueness results for these inverse problems.

\section{Uniqueness for inverse obstacle scattering}\label{o}

Denote by $u_j^s$ and $u_j^\infty$ the scattered field and its far-field pattern, respectively,
associated with the obstacle $D_j\cup B$ and corresponding to the incident field $u^i$, $j=1,2$,
where $B$ is a given, sound-soft reference ball. See Fig. \ref{figo} for the geometry of the scattering problem.
Assume that $k^2$ is not a Dirichlet eigenvalue of $-\Delta$ in $B$, where $k^2$ is called 
a Dirichlet eigenvalue of $-\Delta$ in $B$ if the interior Dirichlet boundary value problem
\[\begin{cases}
\Delta u+k^2u=0 & \text{in}\;\;B,\\
u=0 & \text{on}\;\;\pa B
\end{cases}\]
has a nontrivial solution $u$. Note that if the radius $\rho$ of the ball $B$ satisfies that $k\rho<\pi$ 
then $k^2$ is not a Dirichlet eigenvalue of $-\Delta$ in $B$ (see \cite[Corollary 5.3]{CK}).

\begin{figure}[!ht]
\centering
\includegraphics[width=0.7\textwidth]{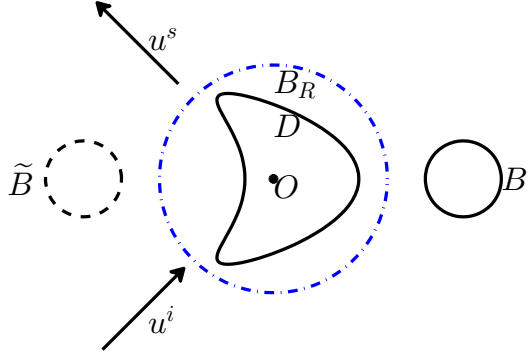}
\vspace{-1cm}
\caption{Scattering by a bounded obstacle.
}\label{figo}
\end{figure}

\begin{theorem}\label{ot}
Assume that $B$ is a given sound-soft reference ball such that $k^2$ is not a Dirichlet eigenvalue
of $-\Delta$ in $B$. Suppose $D_1$ and $D_2$ are two obstacles with $\ov{D_1\cup D_2}\subset B_R$,
where $B_R$ is a ball of radius $R$ and centered at the origin and satisfies that
$\ov{B}\cap\ov{B_R}=\emptyset$. If the corresponding far-field patterns satisfy that
\be\label{or=}
&&|u_1^\infty(\hat x,d)|=|u_2^\infty(\hat x,d)|\;\;\forall\hat x,d\in\Sp^2,\\ \label{om=}
&&|u_1^\infty(\hat x;d,d_0)|=|u_2^\infty(\hat x;d,d_0)|\;\;\forall\hat x,d\in\Sp^2
\en
for an arbitrarily fixed $d_0\in\Sp^2$, then $D_1=D_2$ and $\mathscr{B}_1=\mathscr{B}_2$.
\end{theorem}

\begin{proof}
Arguing similarly as in the proof of (4.9) and (4.10) in \cite{XZZ} (see the proof of Theorem 2.2
in \cite{XZZ}), we can obtain from (\ref{or=}) and (\ref{om=}) that either
\be\label{o1}
u_1^\infty(\hat{x},d)=e^{i\alpha}u_2^\infty(\hat{x},d)\;\;\forall\hat{x},d\in\Sp^2
\en
or
\be\label{o2}
u_1^\infty(\hat{x},d)=e^{i\beta}\ov{u_2^\infty(\hat{x},d)}\;\;\forall\hat{x},d\in\Sp^2
\en
holds for some real constants $\alpha,\beta$ which are independent of $\hat x$ and $d$.

We now show that (\ref{o2}) does not hold. In fact, suppose (\ref{o2}) holds.
By Green's representation for the scattering solution $u_2^s$ in $\R^3\se(\ov{B\cup B_R})$
(see \cite[Theorem 2.5]{CK}) we obtain that for any fixed $d\in\Sp^2$,
\ben
u_2^s(x,d)=\int_{\pa B\cup\pa B_R}\left\{u_2^s(y,d)\frac{\pa\Phi(x,y)}{\pa\nu(y)}
-\frac{\pa u_2^s(y,d)}{\pa\nu(y)}\Phi(x,y)\right\}ds(y),\\
x\in\R^3\se(\ov{B\cup B_R}),
\enn
where $\nu(y)$ is the unit normal vector at $y\in\pa B$ or $y\in\pa B_R$ directed into the exterior 
of $B$ or $B_R$ and $\Phi(x,y)$ is the fundamental solution to the Helmholtz equation in $\R^3$ given by
\ben
\Phi(x,y):=\frac1{4\pi}\frac{e^{ik|x-y|}}{|x-y|},\;\;\;x\not=y,
\enn
and thus the far-field pattern of $u_2^s$ is given as follows (see \cite[(2.14)]{CK}):
\ben
u_2^\infty(\hat x,d)=\frac1{4\pi}\int_{\pa B\cup\pa B_R}\left\{u_2^s(y,d)
\frac{\pa e^{-ik\hat x\cdot y}}{\pa\nu}
-\frac{\pa u_2^s(y,d)}{\pa\nu}e^{-ik\hat x\cdot y}\right\}ds(y),\;\;\hat x\in\Sp^2.
\enn
From this and (\ref{o2}) it follows that for any $\hat{x}\in\Sp^2$,
\ben
u_1^\infty(\hat x,d)&=&e^{i\beta}\ov{u_2^\infty(\hat x,d)}\\
&=&\frac{e^{i\beta}}{4\pi}\int_{\pa B\cup\pa B_R}\left\{\ov{u_2^s}(y,d)
   \frac{\pa e^{ik\hat x\cdot y}}{\pa\nu}
   -\frac{\pa\ov{u_2^s}}{\pa\nu}(y,d)e^{ik\hat x\cdot y}\right\}ds(y)\\
&=&\frac{e^{i\beta}}{4\pi}\int_{\pa\wid{B}\cup\pa B_R}\left\{\ov{u_2^s}(-y,d)
   \frac{\pa e^{-ik\hat x\cdot y}}{\pa\nu}
   -\frac{\pa\ov{u_2^s}}{\pa\nu}(-y,d)e^{-ik\hat x\cdot y}\right\}ds(y),
\enn
where $\wid{B}:=\{x\in\R^3:-x\in B\}$. 
Rellich's lemma gives
\be\no
u_1^s(x,d)=e^{i\beta}\int_{\pa\wid{B}\cup\pa B_R}\left\{\ov{u_2^s}(-y,d)
\frac{\pa\Phi(x,y)}{\pa\nu(y)}-\frac{\pa\ov{u_2^s}}{\pa\nu}(-y,d)\Phi(x,y)\right\}ds(y),\\ \label{ot+}
x\in\R^3\se\ov{\wid{B}\cup B_R}.\;\;
\en
This means that $u_1^s(\cdot,d)$ can be analytically extended into $\R^3\se\ov{\wid{B}\cup B_R}$ and
satisfies the Helmholtz equation in $\R^3\se\ov{\wid{B}\cup B_R}$. On the other hand,
$u_1^s(\cdot,d)$ also satisfies the Helmholtz equation in $\R^3\se\ov{B\cup B_R}$.
Since $\ov{B}\cap\ov{B_R}=\emptyset$, then the origin $O\notin B$ and $B\cap\wid{B}=\emptyset$.
Therefore, $u_1^s(\cdot,d)$ satisfies the Helmholtz equation in $\R^3\se\ov{B_R}$.
Noting that the total field $u_1:=u^i+u_1^s$ satisfies the sound-soft boundary condition 
on $\pa B$, we have
\ben
\begin{cases}
\Delta u_1+k^2u_1=0 & \mbox{in}\;\;\; B,\\
u_1=0 & \mbox{on}\;\;\; \pa B.
\end{cases}
\enn
Since $k^2$ is not a Dirichlet eigenvalue of $-\Delta$ in $B$, we have $u_1\equiv0$ in $B$,
which, together with the analyticity of the total field $u_1$ in $\R^3\se\ov{B_R}$, implies
that $u_1\equiv0$ in $\R^3\se\ov{B_R}$. This is a contradiction, so (\ref{o2}) does not hold.
Therefore, (\ref{o1}) holds.

Now, by (\ref{o1}) and Rellich's lemma we obtain that
\ben
u_1^s(x,d)=e^{i\alpha}u_2^s(x,d)\;\;\forall x\in G,\;d\in\Sp^2,
\enn
where $G$ denotes the unbounded component of the complement of $B\cup D_1\cup D_2$.
By the sound-soft boundary condition on $\pa B$, we have $u^s_j(x,d)=-e^{ikx\cdot d}$
for all $x\in\pa B$, $j=1,2$, which implies that
\[
-e^{ikx\cdot d}=-e^{i\alpha}e^{ikx\cdot d}\;\;\forall x\in\pa B.
\]
Hence $e^{i\alpha}=1$, which means that
\ben
u_1^\infty(\hat x,d)=u_2^\infty(\hat x,d)\;\;\forall\hat{x},\;d\in\Sp^2.
\enn
This, together with \cite[Theorem 5.6]{CK}, implies that $D_1=D_2$ and $\mathscr B_1=\mathscr B_2$.
The proof is thus complete.
\end{proof}

\begin{remark}\label{r1} {\rm
Theorem \ref{ot} remains true for the two-dimensional case, and the proof is similar.
}
\end{remark}

\section{Uniqueness for inverse medium scattering}\label{m}

Assume that $B$ is a given homogeneous medium ball characterized by the positive constant refractive
index $n_0\neq 1$. Assume further that $n_1,n_2\in L^\infty(\R^3)$ are the refractive indices
of two inhomogeneous media with $m_j:=n_j-1$ supported in $\ov{D_j},j=1,2$.
Denote by $u_j^s$ and $u_j^\infty$ the scattered field and its far-field pattern, respectively,
associated with the inhomogeneous medium with the refractive index $\wid{n}_j$ and corresponding
to the incident field $u^i$, $j=1,2$. Here, the refractive index $\wid{n}_j$ is given by
\ben
\wid{n}_j(x):=\begin{cases}
n_0, & x\in B,\\
n_j(x), & x\in\R^3\se\ov{B}
\end{cases}
\enn
for $j=1,2$. See Fig. \ref{figm} for a geometric description of the scattering problem.
Suppose $k^2$ is not an interior transmission eigenvalue of $-\Delta$ in $B$.  
Here, $k^2$ is called an interior transmission eigenvalue of $-\Delta$ in $B$
if the interior transmission problem
\ben
\begin{cases}
\Delta w+k^2n_0w=0\;&\;\text{in}\;\; B,\\
\Delta v+k^2v=0\;&\;\text{in}\;\; B,\\
\ds w=v,\;\;\frac{\pa w}{\pa\nu}=\frac{\pa v}{\pa\nu}\;&\;\text{on}\;\;\pa B
\end{cases}
\enn
has a nontrivial solution $(w,v)$. Note that if the radius $\rho$ of the ball $B$ satisfies
that $\rho<\pi/[2k(\sqrt{n_0}+1)]$ then $k^2$ is not an interior transmission eigenvalue
of $-\Delta$ in $B$ (see \cite[Theorem 4.1]{ZG18}).

\begin{figure}[!ht]
\centering
\includegraphics[width=0.7\textwidth]{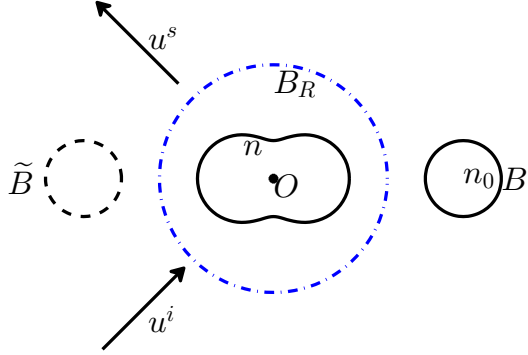}
\vspace{-1cm}
\caption{Scattering by an inhomogeneous medium.
}\label{figm}
\end{figure}

\begin{theorem}\label{mt}
Assume that $B$ is a given homogeneous medium ball with the constant refractive index $n_0\neq 1$
such that $k^2$ is not an interior transmission eigenvalue of $-\Delta$ in $B$.
Assume further that $n_1,n_2\in L^\infty(\R^3)$ are the refractive indices of two inhomogeneous
media with $m_j:=n_j-1$ supported in $\ov{D_j},j=1,2$. Suppose $\ov{D_1\cup D_2}\subset B_R$,
where $B_R$ is a ball of radius $R$ and centered at the origin and satisfies that
$\ov{B}\cap\ov{B_R}=\emptyset$. If the corresponding far field patterns satisfy $(\ref{or=})$
and $(\ref{om=})$ for an arbitrarily fixed $d_0\in\Sp^2$, then $n_1=n_2$.
\end{theorem}

\begin{proof}
We first claim that for any fixed $d_0\in\Sp^2$ we have $u^\infty_j(\cdot,d_0)\not\equiv0$ in $\Sp^2$,
$j=1,2$. Otherwise, $u^\infty_j(\hat{x},d_0)\equiv0$ for any $\hat{x}\in\Sp^2$, and so, by Rellich's
lemma, $u^s_j({x},d_0)\equiv0$ for any ${x}\in\R^3\se(\ov{B\cup D_j})$, $j=1,2$.
Thus, $\Delta u_j+k^2n_0 u_j=0$ in $B$, where $u_j=u^i+u_j^s$ is the total field satisfying that
$u_j=u^i$, $\pa u_j/\pa\nu=\pa u^i/\pa\nu$ on $\pa B$. Note that $\Delta u^i+k^2u^i=0$ in $B$.
Since $k^2$ is not an interior transmission eigenvalue of $-\Delta$ in $B$, we have $u_j=u^i\equiv0$ in $B$.
This is a contradiction, implying that our claim is true.
Thus, by arguing similarly as in the beginning of the proof of Theorem \ref{ot} it follows from (\ref{or=})
and (\ref{om=}) that either (\ref{o1}) or (\ref{o2}) holds.

Again we can prove that (\ref{o2}) does not hold. Suppose (\ref{o2}) holds to the contrary.
Then, by a similar argument as in the proof of (\ref{ot+}) we can obtain that
\be\no
u_1^s(x,d)=e^{i\beta}\int_{\pa\wid{B}\cup\pa B_R}\left\{\ov{u_2^s}(-y,d)
\frac{\pa\Phi(x,y)}{\pa\nu(y)}-\frac{\pa\ov{u_2^s}}{\pa\nu}(-y,d)\Phi(x,y)\right\}ds(y),\\ \label{mt+}
  x\in\R^3\se\ov{\wid{B}\cup B_R}.\;\;
\en
Since $\ov{B}\cap\ov{B_R}=\emptyset$, then $B\cap\wid{B}=\emptyset$ and $B\subset\R^3\se\ov{\wid{B}\cup B_R}$.
Thus, by (\ref{mt+}) we know that the total field $u_1:=u^i+u_1^s$ satisfies that
$\Delta u_1+k^2u_1=0$ in $B$. On the other hand, by the definition of $u_1$ we have that
$\Delta u_1+k^2n_0u_1=0$ in $B$. 
Then, we obtain that $u_1\equiv0$ in $B$.
By (\ref{mt+}) again it is known that $u_1^s$ (and thus $u_1$) is analytic in $\R^3\se\ov{\wid{B}\cup B_R}$.
Thus, and on noting that $B\subset\R^3\se\ov{\wid{B}\cup B_R}$, we obtain by the analytic continuation
that $u_1\equiv0$ in $\R^3\se\ov{\wid{B}\cup B_R}$.
This is a contradiction, implying that (\ref{o2}) does not hold. Hence, (\ref{o1}) holds.

By (\ref{o1}) and Rellich's lemma it follows that
\be\label{os+}
u_1^s(x,d)=e^{i\alpha}u_2^s(x,d)\;\;\forall x\in G,\;d\in\Sp^2,
\en
where $G$ denotes the unbounded component of the complement of $B\cup D_1\cup D_2$.
For any fixed $d\in\Sp^2$, define $u:=u_1-e^{i\alpha}u_2$ and $v:=(1-e^{i\alpha})u^i$ in $\R^3$.
Then, by (\ref{os+}) it is seen that $u=v$ in $G$, and so, $u=v$ and
$\pa u/\pa\nu=\pa v/\pa\nu$ on $\pa B$. On the other hand, by the definition of $u_j$, $j=1,2$,
we have $\Delta u+k^2n_0u=0$ in $B$. As a result, we obtain that
\ben
\begin{cases}
\Delta u+k^2n_0u=0 & \text{in}\;\;B,\\
\Delta v+k^2v=0 & \text{in}\;\;B,\\
u=v,\;\;\pa u/\pa\nu=\pa v/\pa\nu & \text{on}\;\;\pa B.
\end{cases}
\enn
Since $k^2$ is not an interior transmission eigenvalue of $-\Delta$ in $B$, then we must have
$e^{i\alpha}=1$, and so, (\ref{o1}) becomes
\ben
u_1^\infty(\hat x,d)=u_2^\infty(\hat x,d)\;\;\forall\hat{x},\;d\in\Sp^2.
\enn
By this and \cite[Theorem 6.26]{Kirsch} we have $n_1=n_2$. The proof is thus complete.
\end{proof}

\begin{remark}\label{r2} {\rm
Theorem \ref{mt} also holds in two dimensions if the assumption $n_j\in L^\infty(\R^3)$ is
replaced by the condition that $n_j$ is piecewise in $W^{1,p}(D_j)$ with $p>2$, $j=1,2$.
In this case, the proof is similar except that we need Bukhgeim's result in \cite{B} (see also
the theorem in Section 4.1 in \cite{Blasten}) instead of \cite[Theorem 6.26]{Kirsch} in the proof.
}
\end{remark}

\section{Uniqueness for inverse scattering by locally rough surfaces}\label{lrs}

Suppose $h_j\in C^2(\R^2)$ has a compact support in $\R^2$, $j=1,2$. 
Let us denote by $\G_j=\{(x_1,x_2,x_3):x_3=h_j(x_1,x_2),\,(x_1,x_2)\in\R^2\}$ the locally rough surface
and by $D_{j,+}=\{(x_1,x_2,x_3):x_3>h_j(x_1,x_2),\,(x_1,x_2)\in\R^2\}$ the half-space above the locally 
rough surface $\G_j$, $j=1,2$.
Let $B_R:=\{x\in\R^3:|x|<R\}$ be a ball with radius $R$ large enough so that the local
perturbation $\G_{jp}:=\G_j\ba\G_0=\{(x_1,x_2,h_j(x_1,x_2)):(x_1,x_2)\in\mbox{suppt}(h_j)\}\subset B_R$,
$j=1,2$, where $\G_0:=\{(x_1,x_2,x_3)\in\R^3:x_3=0\}$ is the infinite plane. 
Then $\G_{jR}:=\G_j\cap B_R$ represents the part of $\G_j$ containing the local perturbation
$\G_{jp}$ of the infinite plane $\G_0$, $j=1,2$.
Let $B\subset D_{1,+}\cap D_{2,+}$ be a given, sound-soft reference ball (that is, above the locally
rough surfaces $\G_1$ and $\G_2$) satisfying that $k^2$ is not a Dirichlet eigenvalue of $-\Delta$ in $B$. 
See Fig. \ref{figr} for the geometry of the scattering problem.

Denote by $u_j^s$ and $u_j^\infty$ the scattered field and its far-field pattern, respectively,
associated with $\{\G_j,B\}$ and corresponding to the incident plane wave $u^i$, $j=1,2$.
Then we have the following theorem on uniqueness in inverse scattering by locally rough surfaces
with phaseless far-field data.

\begin{figure}[!ht]
  \centering
  \includegraphics[width=0.7\textwidth]{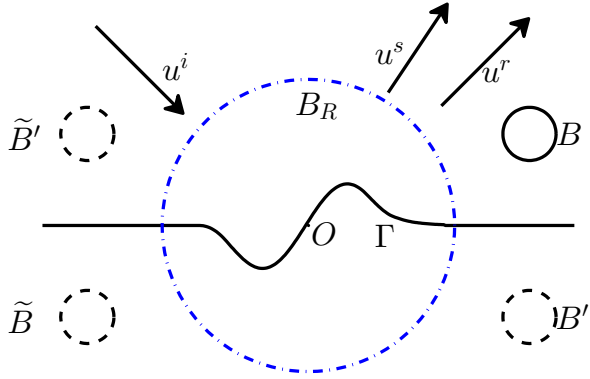}
  \vspace{-1cm}
  \caption{Scattering by a locally rough surface.
  }\label{figr}
\end{figure}

\begin{theorem}\label{lt}
Suppose $\G_j$, $j=1,2,$ are two sound-soft, locally rough surfaces.
Assume that $\ov{B}\cap(\ov{B_R}\cup\{(0,0,x_3):x_3\in\R\})=\emptyset$.
If the corresponding far-field patterns satisfy that
\be\label{r=}
|u_1^\infty(\hat{x},d)|&=&|u_2^\infty(\hat{x},d)|\quad\forall\hat{x}\in\Sp^2_+,\;d\in\Sp^2_-,\\ \label{m=}
|u_1^\infty(\hat{x};d,d_0)|&=&|u_2^\infty(\hat{x};d,d_0)|\quad\forall\hat{x}\in\Sp^2_+,\;d\in\Sp^2_-
\en
for an arbitrarily fixed $d_0\in\Sp^2_-$, then $\G_1=\G_2$.
\end{theorem}

\subsection{Some lemmas}

To prove Theorem \ref{lt} we need some lemmas. In this subsection, we assume that $h\in C^2(\R^2)$ 
has a compact support in $\R^2$. Denote by $\G:=\{(x_1,x_2,x_3):x_3=h(x_1,x_2),(x_1,x_2)\in\R^2\}$ and 
$D_+=\{(x_1,x_2,x_3):x_3>h(x_1,x_2),\,(x_1,x_2)\in\R^2\}$ the locally rough surface 
and the half-space above the locally rough surface $\G$, respectively.
Assume that the local perturbation 
$\G_{p}:=\G\ba\G_0=\{(x_1,x_2,h(x_1,x_2)):(x_1,x_2)\in\mbox{suppt}(h)\}\subset B_R$.
Then $\G_{R}:=\G\cap B_R$ represents the part of $\G$ containing the local perturbation
$\G_{p}$ of the infinite plane $\G_0$.
Assume further that the sound-soft ball $B\subset D_+$ with $\ov{B}\cap\ov{B_R}=\emptyset$.

\begin{lemma}\label{NotEquiv0}
For any fixed $d_0\in\Sp_-$, let $u^\infty(\hat x,d_0)$ denote the far-field pattern 
corresponding to $\{\G,B\}$. Then $u^\infty(\cdot,d_0)\not\equiv0\;\;\text{in}\;\;\Sp_+^2.$
\end{lemma}

\begin{proof}
Suppose to the contrary that $u^\infty(\hat x,d_0)\equiv0$ for all $\hat{x}\in\Sp_+^2.$
Then, by Rellich's lemma (see \cite[Theorem 2.2]{W}) we have
\[u^s(x,d_0)=0\;\;\forall x\in D_+\se\ov{B}.\]
The Dirichlet boundary condition implies that
\[u(x,d_0):=u^i(x,d_0)+u^r(x,d_0)+u^s(x,d_0)=u^i(x,d_0)+u^r(x,d_0)=0\;\;\mbox{on}\;\;\G\cup\pa B.\]
Noting that $\Delta u+k^2u=0$ in $B$ and that $k^2$ is not a Dirichlet eigenvalue of $-\Delta$ in $B$,
we obtain that $u(x,d_0)\equiv0$ or $u^i(x,d_0)\equiv-u^r(x,d_0)$ for $x\in B$. This is a contradiction,
and the proof is thus complete.
\end{proof}

\begin{lemma}\label{extension}
Assume further that $v$ is a solution to the Helmholtz equation $\Delta v+k^2v=0$ in $D_+\se\ov{B}$ 
with the boundary condition $v=0$ on $\G\se\ov{B_R}$.
Let $\wid{v}$ be the extension of $v$ into $D_-\se(\ov{B_R\cup B^\prime})$ by reflection, that is,
\ben
\wid{v}(x)=\left\{\begin{array}{ll}
v(x), & x\in D_+\se\ov{B},\\
-v(x'), & x\in D_-\se(\ov{B_R\cup B^\prime}),
\end{array}\right.
\enn
where $D_-:=\{x\in\R^3:x_3<h(x_1,x_2)\}$, $B^\prime:=\{x\in\R^3:x'\in B\}$ and $x'=(x_1,x_2,-x_3)$.
Then $\wid{v}$ satisfies the Helmholtz equation in $\R^3\se(\ov{B_R\cup B\cup B^\prime})$.
\end{lemma}

\begin{proof}
Note that the local perturbation $\G_p\in B_R$ and $\ov{B}\cap\ov{B_R}=\emptyset$. 
By a regularity argument (see \cite[Theorem 3.1]{W}) and the reflection principle, we know that
$\wid{v}\in C^2(\R^3\se(\ov{B_R\cup B\cup B^\prime})$ and satisfies the Helmholtz equation in
$\R^3\se(\ov{B_R\cup B\cup B^\prime})$ (cf. \cite[Theorem 2.18 (a)]{KG}).
The proof is complete.
\end{proof}

\begin{lemma}\label{reci} 
For any $d\in\Sp_-$, let $u^\infty(\hat x,d)$ denote the far-field pattern corresponding to $\{\G,B\}$.
Then we have the reciprocity relation
\be\label{reci-far}
u^\infty(\hat x,d)=u^\infty(-d,-\hat x)\;\;\;\forall\hat{x}\in\Sp_+^2,\;d\in\Sp_-^2.
\en
\end{lemma}

\begin{proof}
Choose $R'>R$ large enough such that $\ov{B_R\cup B}\subset B_{R'}$.
Extend the total field $u(x,d)=u^i(x,d)+u^r(x,d)+u^s(x,d)$ into $D_-\se\ov{B_{R'}}$ by reflection
as in Lemma \ref{extension}, which is an odd function and denoted again by $u(x,d)$.
Apply Green's second formula to $(u^i+u^r)(y,d)$ and $(u^i+u^r)(y,-\hat{x})$ in $B_{R'}$ to get
\ben
0&=&\int_{\pa B_{R'}}\left[(u^i+u^r)(y,d)\frac{\pa(u^i+u^r)(y,-\hat{x})}{\pa\nu}\right.\\
&&\qquad\qquad\left.-(u^i+u^r)(y,-\hat{x})\frac{\pa(u^i+u^r)(y,d)}{\pa\nu}\right]ds(y).\\ \label{reci-1}
\enn
On the other hand, applying Green's second formula to $u^s(y,d)$ and $u^s(y,-\hat{x})$ in $\R^3\se\ov{B_{R'}}$
and making use of the Sommerfeld radiation condition for the scattered field $u^s$ we obtain that
\be\label{reci-2}
0=\int_{\pa B_{R'}}\left[u^s(y,d)\frac{\pa u^s(y,-\hat{x})}{\pa\nu}
-u^s(y,-\hat{x})\frac{\pa u^s(y,d)}{\pa\nu}\right]ds(y).
\en
Recalling the definition of $u^i$ and $u^r$ it follows from \cite[Theorem 2.6]{CK} that
\ben
4\pi u^\infty(\hat{x},d)&=&\int_{\pa B_{R'}}\left[u^s(y,d)\frac{\partial u^i}{\partial\nu}(y,-\hat x)
-u^i(y,-\hat x)\frac{\partial u^s}{\partial\nu}(y,d)\right]ds(y)\\
&=&\int_{\pa B_{R'}}\left[u^s(y,d)\frac{\partial u^r}{\partial\nu}(y,-\hat x)
-u^r(y,-\hat x)\frac{\partial u^s}{\partial\nu}(y,d)\right]ds(y),
\enn
which implies that
\ben
4\pi u^\infty(\hat{x},d)&=&\frac12\int_{\pa B_{R'}}\left[u^s(y,d)\frac{\pa(u^i+u^r)}{\pa\nu}(y,-\hat x)\right.\\
&&\qquad\qquad\left.-(u^i+u^r)(y,-\hat x)\frac{\pa u^s}{\pa\nu}(y,d)\right]ds(y).\\ \label{reci-3}
\enn
Interchanging the role of $\hat x$ and $d$ in the above equation gives
\ben
4\pi u^\infty(-d,-\hat x)&=&\frac12\int_{\pa B_{R'}}\left[u^s(y,-\hat x)\frac{\pa(u^i+u^r)}{\pa\nu}(y,d)\right.\\
&&\qquad\qquad\left.-(u^i+u^r)(y,d)\frac{\pa u^s}{\pa\nu}(y,-\hat x)\right]ds(y).\label{reci-4}
\enn
Combining (\ref{reci-1})-(\ref{reci-4}) leads to the result
\ben
4\pi[u^\infty(\hat x,d)-u^\infty(-d,-\hat x)]
&=&\frac12\int_{\pa B_{R'}}\left[u(y,d)\frac{\pa u}{\pa\nu}(y,-\hat{x})\right.\\
&&\left.\qquad\qquad-u(y,-\hat x)\frac{\pa u}{\pa\nu}(y,d)\right]ds(y).
\enn
Since $u$ is odd, the integrand of the integral term on the right-hand side of the above equation is even
and so,
\ben
4\pi[u^\infty(\hat x,d)-u^\infty(-d,-\hat x)]&=&\int_{\pa B_{R'}^+}\left[u(y,d)\frac{\pa u}{\pa\nu}(y,-\hat{x})
-u(y,-\hat x)\frac{\pa u}{\pa\nu}(y,d)\right]ds(y)\\
&=&\int_{\pa B\cup\G_{R'}}\left[u(y,d)\frac{\pa u}{\pa\nu}(y,-\hat{x})
-u(y,-\hat x)\frac{\pa u}{\pa\nu}(y,d)\right]ds(y),
\enn
where $\pa B_{R'}^+:=\{x\in \pa B_{R'}:x_3>0\}$, $\G_{R'}:=\G\cap B_{R'}$ and Green's second formula is applied to
$u(y,d)$ and $u(y,-\hat{x})$ in $D_+\cap{B_{R'}}$ to obtain the last equality.
The required equation (\ref{reci-far}) follows from the above equation on using the boundary condition
$u(y,d)=u(y,-\hat x)=0$ for $y\in\G\cup\pa B$. The lemma is thus proved.
\end{proof}

\subsection{Proof of Theorem \ref{lt}}

By making use of Lemmas \ref{NotEquiv0} and \ref{reci} and arguing similarly as in the proof of (4.9)
and (4.10) in \cite{XZZ}, it can be obtained from (\ref{r=}) and (\ref{m=}) that
\be\label{case1}
u_1^\infty(\hat x,d)=e^{i\alpha}u_2^\infty(\hat x,d)\;\;\forall\hat x\in\Sp_+^2,\;\;d\in\Sp_-^2
\en
or
\be\label{case2}
u_1^\infty(\hat x,d)=e^{i\beta}\ov{u_2^\infty(\hat x,d)}\;\;\forall\hat x\in\Sp_+^2,\;\;d\in\Sp_-^2,
\en
where $\alpha$ and $\beta$ are real constants independent of $\hat x$ and $d$.

We now prove that (\ref{case2}) dose not hold. To this end, we consider the extension $\wid{u}_j^s$ of $u_j^s$ 
defined by Lemma \ref{extension}, $j=1,2$. For arbitrarily fixed $d\in\Sp^2_-$, by Green's representation 
for the scattering solution $\wid{u}_2^s$ in $\R^3\se\ov{B_R\cup B\cup B'}$ (see \cite[Theorem 2.5]{CK}) 
we obtain that 
\ben
\wid{u}_2^s(x,d)=\int_{\pa B_R\cup\pa B\cup\pa B'}\left[\wid{u}_2^s(y,d)\frac{\pa\Phi(x,y)}{\pa\nu(y)}
-\frac{\pa\wid{u}_2^s(y,d)}{\pa\nu(y)}\Phi(x,y)\right]ds(y)\\
\qquad\forall x\in\R^3\se\ov{B_R\cup B\cup B'}.
\enn
Thus the far-field pattern ${u}^\infty_2$ of $u_2^s$ is given by
\ben
{u}_2^\infty(\hat{x},d)=\wid{u}_2^\infty(\hat{x},d)=\frac1{4\pi}\int_{\pa B_R\cup\pa B\cup\pa B'}
\left[\wid{u}_2^s(y,d)\frac{\pa e^{-ik\hat{x}\cdot y}}{\pa\nu}
-\frac{\pa\wid{u}_2^s}{\pa\nu}(y,d)e^{-ik\hat x\cdot y}\right]ds(y)\\
\quad\forall\hat{x}\in\Sp^2_+.
\enn
From (\ref{case2}) it follows that
\ben
u_1^\infty(\hat{x},d)&=&e^{i\beta}\ov{{u}_2^\infty(\hat{x},d)}\\
&=&\frac{e^{i\beta}}{4\pi}\int_{\pa B_R\cup\pa B\cup\pa B'}
\left[\ov{\wid{u}_2^s}(y,d)\frac{\pa e^{ik\hat x\cdot y}}{\pa\nu}
-\frac{\pa\ov{\wid{u}_2^s}}{\pa\nu}(y,d)e^{ik\hat x\cdot y}\right]ds(y)\\
&=&\frac{e^{i\beta}}{4\pi}\int_{\pa B_R\cup\pa\wid{B}\cup\pa\wid{B}'}
\left[\ov{\wid{u}_2^s}(-y,d)\frac{\pa e^{-ik\hat x\cdot y}}{\pa\nu}
-\frac{\pa\ov{\wid{u}_2^s}}{\pa\nu}(-y,d)e^{-ik\hat x\cdot y}\right]ds(y)
\enn
where $\wid{B}':=\{x\in\R^3:-x\in B'\}$. By Rellich's lemma (see \cite[Theorem 2.2]{W}) we have
\ben
{u}_1^s(x,d)=e^{i\beta}\int_{\pa B_R\cup\pa\wid{B}\cup\pa\wid{B}'}\left[\ov{\wid{u}_2^s}(-y,d)
\frac{\pa\Phi(x,y)}{\pa\nu(y)}-\frac{\pa\ov{\wid{u}_2^s}}{\pa\nu}(-y,d)\Phi(x,y)\right]ds(y)\\
\forall x\in\R^3_+\se\ov{B_R\cup\wid{B}\cup\wid{B}'}.
\enn
This means that ${u}_1^s(\cdot,d)$ satisfies the Helmholtz equation in $\R^3_+\se\ov{\wid{B}\cup\wid{B}'}$ 
and is analytic in $\R^3_+\se\ov{\wid{B}\cup\wid{B}'}$. 
On the other hand, ${u}_1^s(\cdot,d)$ satisfies the Helmholtz equation in $\R^3_+\se\ov{B\cup B_R}$. 
Since $\ov{B}\cap(\ov{B_R}\cup\{(0,0,x_3):x_3\in\R\})=\emptyset$, $\ov{({B}\cup B_R)\cap\wid{B}'}=\emptyset$. 
Therefore, ${u}_1^s(\cdot,d)$ satisfies the Helmholtz equation in $\R^3_+\se\ov{B_R}$. 
Thus the total field $u_1=u^i+u^r+u_1^s$ satisfies that
\[\begin{cases}
\Delta u_1+k^2u_1=0 & \text{in}\;\; B,\\
u_1=0 & \text{on}\;\;\pa B.
\end{cases}\]
Since $k^2$ is not a Dirichlet eigenvalue of $-\Delta$ in $B$, we have $u_1\equiv 0$ in $B$, which,
together with the analyticity of the total field $u_1$ in $\R^3_+\se\ov{B_R}$,
implies that $u_1\equiv0$ in $\R^3_+\se\ov{B_R}$. This is a contradiction,
leading to the fact that (\ref{case2}) is impossible. Thus, (\ref{case1}) is true.

Now, by (\ref{case1}) and Rellich's lemma (see \cite[Theorem 2.2]{W}) we have
\be\label{case1+1}
u_1^s(x,d)=e^{i\alpha}u_2^s(x,d)\;\;\forall x\in\R^3_+\se\ov{B\cup B_R},\;\;d\in\Sp_-^2.
\en
From the boundary condition on $\pa B$ it follows that $u_j^s(x,d)=-[u^i(x,d)+u^r(x,d)]$
for all $x\in\pa B,\;d\in\Sp_-^2$, $j=1,2.$ This, together with (\ref{case1+1}), yields
that $e^{i\alpha}=1$, leading to the fact that
\be\label{case1+2}
u_1^\infty(\hat x,d)=u_2^\infty(\hat x,d)\;\;\forall\hat x\in\Sp_+^2,\;\;d\in\Sp_-^2.
\en
By (\ref{case1+2}) and Theorem 4.1 in \cite{ZZ13} it follows that $\G_1=\G_2$.
This completes the proof of Theorem \ref{lt}.

Similarly, we can prove the following theorem for the Neumann boundary condition case.

\begin{theorem}\label{thm4-n}
Assume that $\G_j$, $j=1,2,$ are two sound-hard, locally rough surfaces.
Assume that $\ov{B}\cap(\ov{B_R}\cup\{(0,0,x_3):x_3\in\R\})=\emptyset$.
If the corresponding far-field patterns satisfy that
\ben
|u_1^\infty(\hat{x};d)|&=&|u_2^\infty(\hat{x};d)|\quad\forall\hat{x}\in\Sp^2_+,\;d\in\Sp^2_-,\\
|u_1^\infty(\hat{x};d,d_0)|&=&|u_2^\infty(\hat{x};d,d_0)|\quad\forall\hat{x}\in\Sp^2_+,\;d\in\Sp^2_-
\enn
for an arbitrarily fixed $d_0\in\Sp^2_-$, then $\G_1=\G_2$.
\end{theorem}

\begin{remark}\label{r3} {\rm
Theorems \ref{lt} and \ref{thm4-n} also hold in the two-dimensional case, and the proof is similar.
}
\end{remark}

\section{Conclusion}\label{con}

In this paper, we have proved uniqueness results for inverse obstacle and medium scattering with
phaseless far-field data generated by infinitely many sets of superpositions of two plane waves
with different incident directions at a fixed frequency, without knowing the physical property of the obstacle
and the inhomogeneous medium in advance. Our method is based on putting a known, reference ball into the
scattering system in conjunction with a simple technique of using Green's representation formula
for the scattering solutions and Rellich's lemma. These results improved our previous results
in \cite{XZZ}, where similar uniqueness results have been obtained under certain a priori assumptions
on the physical property of the obstacle and the inhomogeneous medium.
We have further used our method to prove uniqueness in determining a sound-soft or sound-hard
locally rough surface from the phaseless far-field patterns corresponding to infinitely many sets
of superpositions of two plane waves with different incident directions at a fixed frequency.
Our results are also valid in the two-dimensional case.
It is interesting to prove the same uniqueness results without the reference ball.
For the locally rough surface case, our uniqueness results require to know the property of the surface,
and our analysis relies on the reflection principle. Thus, it is interesting to establish similar results
for the impedance and penetrable surface cases.

\section*{Acknowledgements}

This work is partly supported by the NNSF of China grants 91630309, 11501558 and 11571355
and the National Center for Mathematics and Interdisciplinary Sciences, CAS.

\end{document}